\documentclass{amsart}
\usepackage{amsfonts,amssymb,amsmath,amsthm}
\usepackage{url}
\usepackage{enumerate}

\usepackage{amsmath,amssymb,amsfonts,amsbsy}
\usepackage{tikz-cd}
\usepackage{amsbsy}
\usepackage{amscd}
\usepackage{wasysym}
\usetikzlibrary{matrix,arrows,decorations.pathmorphing}
\usepackage{graphicx}
\usepackage{float}

\newtheorem{thrm}{Theorem}[section]
\newtheorem{lem}[thrm]{Lemma}
\newtheorem{prop}[thrm]{Proposition}
\newtheorem{cor}[thrm]{Corollary}
\theoremstyle{definition}

\newtheorem{remark}[thrm]{Remark}
\numberwithin{equation}{section}

\newcommand{\mc}[1]{\mathcal{#1}}
\newcommand{\e}[1]{\emph{#1}}

\newcommand{\la}{\langle}
\newcommand{\ra}{\rangle}
\newcommand{\tr}{\mathrm{tr}}
\newcommand{\rmv}[1]{}

\newcommand{\hs}{\hskip10pt}

\newcommand{\LO}{L^1(G)}
\newcommand{\LOQ}{L^1(\mathbb{G})}
\newcommand{\LOQs}{L^1_*(\mathbb{G})}

\newcommand{\LTQ}{L^2(\mathbb{G})}

\newcommand{\LI}{L^{\infty}(G)}
\newcommand{\LIQ}{L^{\infty}(\mathbb{G})}

\newcommand{\LIQH}{L^{\infty}(\widehat{\mathbb{G}})}

\newcommand{\BH}{\mc{B}(H)}

\newcommand{\BLTQ}{\mc{B}(L^2(\mathbb{G}))}

\newcommand{\TCQ}{\mc{T}(L^2(\mathbb{G}))}

\newcommand{\vphi}{\varphi}

\newcommand{\Lphi}{\Lambda_\varphi}
\newcommand{\Lphis}{\Lambda_{\varphi\otimes\varphi}}

\newcommand{\al}{\alpha}
\newcommand{\be}{\beta}
\newcommand{\lm}{\lambda}
\newcommand{\Lm}{\Lambda}
\newcommand{\Gam}{\Gamma}
\newcommand{\om}{\omega}

\newcommand{\ten}{\otimes}
\newcommand{\oten}{\overline{\otimes}}
\newcommand{\hten}{\widehat{\otimes}}

\newcommand{\id}{\textnormal{id}}
\newcommand{\h}[1]{\widehat{#1}}

\newcommand{\Irr}{\mathrm{Irr}(\mathbb{G})}

\providecommand{\norm}[1]{\lVert#1\rVert}

\newcommand{\G}{\mathbb{G}}
\newcommand{\D}{\mathcal{D}}

\newcommand{\C}{\mathbb{C}}

\newcommand{\R}{\mathbb{R}}

\author{Mahmood Alaghmandan}
\address{Department of Mathematical Sciences, Chalmers University of Technology and University of Gothenburg, Gothenburg SE-412 96, Sweden}
\email{mahala@chalmers.se}

\author{Jason Crann}
\address{School of Mathematics and Statistics, Carleton University, Ottawa, ON, Canada K1S 5B6}
\email{jason.crann@carleton.ca}

\begin{document}

\title[Character density in compact quantum groups]{Character density in central subalgebras of compact quantum groups}

\begin{abstract} We investigate quantum group generalizations of various density results from Fourier analysis on compact groups. In particular, we establish the density of characters in the space of fixed points of the conjugation action on $\LTQ$, and use this result to show the weak* density of characters in $Z\LIQ$. At the level of $\LOQ$, we show that the center $\mc{Z}(\LOQ)$ is precisely the closed linear span of the quantum characters for a large class of compact quantum groups, including arbitrary compact Kac algebras. In the latter setting, we show, in addition, that $\mc{Z}(\LOQ)$ is a completely complemented $\mc{Z}(\LOQ)$-submodule of $\LOQ$.
\end{abstract}
\maketitle

\section{Introduction}

As in the group setting, irreducible characters play a significant role in harmonic analysis on compact quantum groups \cite{Teo,B,CFY,Wo}. In this note, we investigate the relationship between the irreducible characters of compact quantum groups $\G$ and the central subalgebras of the Banach algebras $\LTQ$ and $\LOQ$, and the von Neumann algebra $\LIQ$. We characterize the fixed points of the two canonical conjugation actions on $\LTQ$ as the closed linear span of the characters and quantum characters, respectively, the latter being equal to the center $\mc{Z}(\LTQ)$. We then use these characterizations to establish the weak* density of characters in $Z\LIQ:=\{x\in\LIQ\mid\Gam(x)=\Sigma\Gamma(x)\}$ for arbitrary compact quantum groups $\G$. Note that the norm density of characters in $ZC(\G):=\{x\in C(\G)\mid\Gam(x)=\Sigma\Gam(x)\}$ for compact Kac algebras was established by Lemeux  \cite[Theorem 1.4]{Lemeux} (partially answering an open question of Woronowicz (see \cite[Proposition 5.11]{Wo})). For any compact quantum group whose dual has the central almost completely positive approximation property in the sense of \cite[Definiton 3]{CFY}, we show that $\mc{Z}(\LOQ)$ is the closed linear span of the characters. We establish the same result for arbitrary compact Kac algebras by showing that $\mc{Z}(\LOQ)$ is a completely complemented $\mc{Z}(\LOQ)$-submodule of $\LOQ$.

\section{Compact quantum groups}

A \textit{locally compact quantum group} is a quadruple $\G=(\LIQ,\Gam,\vphi,\psi)$, where $\LIQ$ is a Hopf-von Neumann algebra with a co-associative co-multiplication $\Gam:\LIQ\rightarrow\LIQ\oten\LIQ$, and $\vphi$ and $\psi$ are fixed (normal faithful semifinite) left and right Haar weights on $\LIQ$, respectively \cite{KV1,KV2}. For every locally compact quantum group $\G$ there exists a \e{left fundamental unitary operator} $W$ on $L^2(\G,\vphi)\ten_2 L^2(\G,\vphi)$ and a \e{right fundamental unitary operator} $V$ on $L^2(\G,\psi)\ten_2 L^2(\G,\psi)$ implementing the co-multiplication $\Gam$ via
\begin{equation*}\Gam(x)=W^*(1\ten x)W=V(x\ten 1)V^*, \ \ \ x\in\LIQ.\end{equation*}
Both unitaries satisfy the \e{pentagonal relation}; that is,
\begin{equation}\label{penta}W_{12}W_{13}W_{23}=W_{23}W_{12}\hs\hs\mathrm{and}\hs\hs V_{12}V_{13}V_{23}=V_{23}V_{12}.\end{equation}
At the level of the Hilbert spaces,
$$W^*\Lphis(x\ten y)=\Lphis(\Gam(y)(x\ten 1))\hs\mathrm{and}\hs V\Lambda_{\psi\ten\psi}(a\ten b)=\Lambda_{\psi\ten\psi}(\Gam(a)(1\ten b))$$
for $x,y\in\mc{N}_\vphi$ and $a,b\in\mc{N}_\psi$. By \cite[Proposition 2.11]{KV2}, we may identify $L^2(\G,\vphi)$ and $L^2(\G,\psi)$, so we will simply use $\LTQ$ for this Hilbert space throughout the paper. The \textit{reduced quantum group $C^*$-algebra} of $\LIQ$ is defined as
$$C_0(\G):=\overline{\{(\id\ten\om)(W)\mid \om\in\TCQ\}}^{\norm{\cdot}}.$$
We say that $\G$ is \textit{compact} if $C_0(\G)$ is a unital $C^*$-algebra, in which case we denote $C_0(\G)$ by $C(\G)$. For compact quantum groups it follows that $\vphi$ is finite and right invariant. In particular, $\vphi=\psi$.

We let $R$ and $(\tau_t)_{t\in\R}$ denote the \textit{unitary antipode} and \textit{scaling group} of $\G$, respectively. The unitary antipode satisfies
\begin{equation}\label{e:antipode} (R\ten R)\circ\Gam = \Sigma\circ\Gamma\circ R,\end{equation}
where $\Sigma:\LIQ\oten\LIQ\rightarrow\LIQ\oten\LIQ$ denotes the flip map. The \textit{antipode} of $\G$ is $S=R\tau_{-\frac{i}{2}}$, and is a closed densely defined operator on $\LIQ$, whose domain we denote by $\mc{D}(S)$.

Let $\LOQ$ denote the predual of $\LIQ$. Then the pre-adjoint of $\Gam$ induces an associative completely contractive multiplication on $\LOQ$, defined by
\begin{equation*}\star:\LOQ\hten\LOQ\ni f\ten g\mapsto f\star g=\Gam_*(f\ten g)\in\LOQ.\end{equation*}
There is a canonical $\LOQ$-bimodule structure on $\LIQ$, given by
\begin{equation*}\la f\star x,g\ra=\la x,g\star f\ra\hs\mathrm{and}\hs\la x\star f,g\ra=\la x,f\star g\ra, \ \ \ x\in\LIQ, \ f,g\in\LOQ.\end{equation*}
We say that $\G$ is \textit{co-amenable} if $\LOQ$ has a bounded left (equivalently, right or two-sided) approximate identity (cf. \cite[Theorem 3.1]{BT}).

Let $\LOQs$ be the subspace of $\LOQ$ defined by
\begin{equation*}\LOQs=\{f\in\LOQ : \exists \ g\in\LOQ\hs\text{s.t.}\hs g(x)=f^*\circ S(x)\hs\forall x\in\mc{D}(S)\},\end{equation*}
where $f^*(x)=\overline{f(x^*)}$, $x\in\LIQ$. It is known from \cite[\S 1.13]{V} that $\LOQs$ is a dense subalgebra of $\LOQ$. There is an involution on $\LOQs$ given by $f^o=f^*\circ S$, such that $\LOQs$ becomes a Banach *-algebra under the norm $\norm{f}_*=\text{max}\{\norm{f},\norm{f^o}\}$.

A \textit{unitary co-representation} of $\G$ is a
unitary $U\in\LIQ\oten\BH$ satisfying $(\Gam\ten\id)(U)=U_{13}U_{23}$. Every unitary co-representation gives rise to a representation of $\LOQ$ via
$$\LOQ\ni f \mapsto (f\ten\id)(U)\in\BH.$$
In particular, the left fundamental unitary $W$ gives rise to the \textit{left regular representation} $\lm:\LOQ\rightarrow\BLTQ$ defined by $\lm(f)=(f\ten\id)(W)$, $f\in\LOQ$,
which is an injective, completely contractive homomorphism from $\LOQ$ into $\BLTQ$. Then $L^{\infty}(\h{\G}):=\{\lm(f) : f\in\LOQ\}''$ is the von Neumann algebra associated with the dual quantum group $\h{\G}$ of $\G$. When $\G$ is compact with normalized Haar state $\vphi$, the following holds \cite{Wo}: every irreducible co-representation $u^{\alpha}$ is finite-dimensional and is unitarily equivalent to a sub-representation
of $W$, and every unitary co-representation of $\G$ can be decomposed into a direct
sum of irreducible co-representations. We let $\Irr:=\{u^{\alpha}\}$ denote a complete set of representatives
of irreducible co-representations of $\G$ which are pairwise inequivalent. Slicing by vector functionals $\om_{ij}=\om_{e_j,e_i}$ relative to an orthonormal basis
of $H_\alpha$, we obtain elements $u^{\alpha}_{ij}=(\id\ten\om_{ij})(u^{\alpha})\in\LIQ$ satisfying
\begin{equation*}\Gam(u^{\alpha}_{ij})=\sum_{k=1}^{n_\alpha}u^{\alpha}_{ik}\ten u^{\alpha}_{kj}, \ \ \ 1\leq i,j\leq n_{\alpha}\end{equation*}
The linear space $\mc{A}:=\textnormal{span}\{u^{\alpha}_{ij}\mid\al\in\Irr \ 1\leq i,j\leq n_\al\}$ forms unital Hopf *-algebra which is dense in $C(\G)$.

For every $\alpha\in\Irr$ there exists a positive invertible matrix
$F^\alpha\in M_{n_\alpha}(\C)$ such that the corresponding ``$F$--matrices'' implement the left Haar weight of the dual $\h{\G}$. Without loss of generality, we may assume that $F^{\alpha}=\text{diag}(\lm_1^{\alpha},\cdots,\lm_{n_\alpha}^{\alpha})$ \cite[Proposition 2.1]{D1}. Since $\tr(F^{\alpha})=\tr(F^{\alpha})^{-1}$, it follows that
$$\sum_{i=1}^{n_\al}\lm^\al_i=\sum_{i=1}^{n_\al}\frac{1}{\lm^\al_i}=\tr(F^\al)=:d_\al,$$
where $d_\al$ is the \emph{quantum dimension} of $u^\al$. If $\G$ is a \textit{compact Kac algebra}, meaning $\vphi$ is tracial, then $d_\al=n_\al$ and $F^\al=1_{n_\al}$ for all $\al\in\Irr$. For every $\al$ there exists a conjugate representation $\overline{\al}$ on $\overline{H}_\al$, such that $\lm^{\overline{\al}}_i=(\lm^\al_i)^{-1}$ and
$u_{ij}^{\overline{\al}}=\sqrt{\frac{\lm_i^\al}{\lm_j^\al}}u_{ij}^{\al^*}$ (see \cite[Proposition 1.4.6]{NT2}).



In the general setting the Peter--Weyl orthogonality relations are as follows:
$$\vphi((u^\be_{kl})^*u_{ij}^\al)=\delta_{\al\be}\delta_{ik}\delta_{jl}\frac{1}{\lm^\al_id_\alpha},\ \ \vphi(u^\be_{kl}(u_{ij}^\al)^*)=\delta_{\al\be}\delta_{ik}\delta_{jl}\frac{\lm^\al_j}{d_\alpha}.$$
From this it follows that $\{\sqrt{d_\al \lm_i^{\al}}\Lphi(u_{ij}^\al)\mid\al\in\Irr, 1\leq i,j\leq n_\al\}$ is an orthonormal basis for $\LTQ$.

For an element $x\in\LIQ$, we let $x\cdot\vphi$ and $\vphi\cdot x$ denote the elements in $\LOQ$ given by $\la x\cdot\vphi,y\ra=\vphi(yx)$, and
$\la\vphi\cdot x,y\ra=\vphi(xy)$, $y\in\LIQ$. If $x=u_{ij}^\al$, we denote $u_{ij}^\al\cdot\vphi$ by $\vphi_{ij}^\al$. By the density of $\mc{A}$ in $C(\G)$, it follows that $\LIQ\cdot\vphi$ is dense in $\LOQ$. Let
$$\mathcal{I}_\vphi:=\{f\in\LOQ\mid \ \exists  \ M>0 \ : \ |\la f,x^*\ra|\leq M\norm{\Lm_\vphi(x)} \ \forall \ x\in\LIQ\}.$$
Then $\mathcal{I}_\vphi$ is a dense left ideal in $\LOQ$ containing $\LIQ\cdot\vphi$ such that for every $f\in\mathcal{I}_\vphi$, there exists a unique $a(f)\in\LTQ$ satisfying $\la a(f),\Lm_\vphi(x)\ra=\la f,x^*\ra$ for all $x\in\LIQ$.
If $x\in\LIQ$, then $a(x\cdot\vphi)=\Lm_\vphi(x)$. 

By left invariance of the Haar state, the map $I:\LTQ\ni\Lphi(x)\mapsto \Lphis(\Gam(x))\in\LTQ\ten_2\LTQ$ is an isometry. Composing its adjoint $I^*:\LTQ\ten_2\LTQ\rightarrow\LTQ$ with the canonical contraction $\LTQ\ten^{\gamma}\LTQ\rightarrow \LTQ\ten_2\LTQ$, where $\ten^\gamma$ denotes the projective tensor product, we obtain a Banach algebra structure on $\LTQ$. On elementary tensors, the multiplication is given by
$$\Lm_\vphi(x)\ten\Lphi(y)\mapsto a((x\cdot\vphi)\star(y\cdot\vphi)), \ \ \ x,y\in\LIQ.$$
Moreover, there exists a contractive homomorphic injection $b:\LTQ\rightarrow\LOQ$ satisfying $b(a(f))=f$ for all $f\in\mathcal{I}_\vphi$. We refer the reader to \cite[\S6.2]{ES} for details in the Kac case; the proofs carrying over verbatim to general compact $\G$. 

As in the case of compact groups, the irreducible characters of $\G$ play an important role in the harmonic analysis. For $\al\in\Irr$, we let
$$\chi^\al:=(\id\ten\tr)(u^{\al})=\sum_{i=1}^{n_\al}u^\al_{ii}\in\LIQ$$
be the \emph{character} of $\al$, and we let
$$\chi^\al_q:=(\id\ten F^\al)(u^{\al})=\sum_{i=1}^{n_\al}\lm^\al_i u^\al_{ii}\in\LIQ$$
be the \emph{quantum character} of $\al$. The characters (as well as the quantum characters) satisfy the decomposition relations:
\begin{equation}\label{e:N}\chi^\al\chi^\be=\sum_{\gamma\in\Irr}N^{\gamma}_{\al\be}\chi^\gamma,\end{equation}
where $N^{\gamma}_{\al\be}$ is the multiplicity of $\gamma$ in the tensor product representation $\al\ten\be$ (see \cite[Proposition 1.4.3]{NT2}). It follows that $\chi^{\overline{\al}}=\chi^{\al^*}$, $\al\in\Irr$, so that $\overline{\mathrm{span}}\{\chi^\al\mid\al\in\Irr\}$ is a $C^*$-subalgebra of $C(\G)$.
Letting $\vphi^\al_q:=\chi^\al_q\cdot\vphi$ be the $\LOQ$ elements corresponding to the quantum characters of $\G$, it follows from the orthogonality relations that
$$\la\vphi_q^\al\star f,u^{\be*}_{kl}\ra=\la f\star\vphi_q^\al,u^{\be*}_{kl}\ra=\la f,u^{\be*}_{kl}\ra\frac{\delta_{\al\be}}{d_\al}$$
for all $f\in\LOQ$ and $\be\in\Irr$. In particular,
\begin{equation}\label{idem}\vphi_q^\al\star\vphi_q^\al=\frac{1}{d_\al}\vphi_q^\al, \ \ \ \al\in\Irr.\end{equation}
By weak* density of $\mc{A}$ in $\LIQ$ it follows that
$\overline{\text{alg}}\{\vphi_q^\al\mid\al\in\Irr\}$ is a closed ideal in $\mathcal{Z}(\LOQ)$, the center of $\LOQ$. Below we establish the reverse inclusion for a large class of compact quantum groups.

From the Peter--Weyl relations one easily sees that
\begin{equation}\label{eq:inner-on-characters}
\la\Lm_\vphi(\chi^\al),\Lm_\vphi(\chi^\be)\ra=\delta_{\al\be}=\la\Lm_\vphi(\chi_q^\al),\Lm_\vphi(\chi_q^\be)\rangle, \ \ \ \al,\be\in\Irr,
\end{equation}
so spatially, there is no difference between the subspaces of $\LTQ$ generated by $\{\Lm_\vphi(\chi^\al)\}$ and $\{\Lm_\vphi(\chi_q^\al)\}$. The multiplicative structure of these spaces is quite different, however, as we now investigate.

\section{Central Subalgebras}

Let $\G$ be a compact quantum group, and let $\beta_2 :\LOQ\rightarrow\BLTQ$ be the conjugation representation of $\G$, defined by
$$\beta_2 (f)=(f\ten\id)(W(1\ten U^*)W(1\ten U)), \ \ \ f\in\LOQ,$$
where $U=\h{J}J$, and $J$ and $\h{J}$ are the conjugate linear isometries arising from the GNS constructions of $\vphi$ and $\h{\vphi}$, respectively. One can easily verify that
$$W(1\ten U^*)W(1\ten U)\in\LIQ\oten\BLTQ$$
is a unitary co-representation of $\G$, so that $\beta_2 $ is indeed a homomorphism. Let $Z\LTQ$ denote the set of fixed vectors under $\beta_2 $, i.e., those $\xi\in\LTQ$ satisfying $\beta_2 (f)\xi=\la f,1\ra\xi$ for all $f\in\LOQ$. We call $Z\LTQ$ the space
of \textit{central vectors} of $\LTQ$, and in what follows we study its connection to the center $\mc{Z}(\LTQ)$. We begin with a few lemmas.

\begin{lem}\label{l:conj-formula} Let $\G$ be a compact quantum group. Then for $x\in\mc{A}$,
$$\beta_2(\vphi)\Lphi(x)=\sum_{i=1}^n\lm(y_i\cdot \vphi)\Lphi(x_i)=\sum_{i=1}^n\Lphi((\vphi\ten\id)(y_i\ten 1)\Gam(x_i))$$
where $\Gamma(x)=\sum_{i=1}^nx_i\ten y_i$.
\end{lem}

\begin{proof} We let $\sigma:\LTQ\ten\LTQ\rightarrow\LTQ\ten\LTQ$ denote the flip map. If $a\in\mc{A}$, we have
\begin{align*}\la\beta_2(\vphi)\Lphi(x),\Lphi(a)\ra&=\la W\sigma V\sigma\Lphis(1\ten x),\Lphis(1\ten a)\ra\\
&=\la\sigma\Lphis(\Gamma(x)),\Lphis(\Gamma(a))\ra\\
&=\vphi\ten\vphi(\Gamma(a^*)\Sigma\Gamma(x))\\
&=\vphi\ten\vphi(\Gamma(a^*)\Sigma\Gamma(x))\\
&=\sum_{i=1}^n \vphi\ten\vphi(\Gam(a^*)y_i\ten x_i)\\
&=\sum_{i=1}^ny_i\cdot\vphi((\id\ten\vphi)(\Gam(a^*)(1\ten x_i))).
\end{align*}
Since $S((\id\ten\vphi)(\Gam(a^*)(1\ten x_i)))=(\id\ten\vphi)((1\ten a^*)\Gam(x_i))$, and $S(S(y)^*)^*=y$ for all $y\in\mc{A}$, we have
$(\id\ten\vphi)(\Gam(a^*)(1\ten x_i))=S((\id\ten\vphi)(\Gam(x_i^*)(1\ten a)))^*$. Continuing,
\begin{align*}
\la\beta_2(\vphi)\Lphi(x),\Lphi(a)\ra&=\sum_{i=1}^ny_i\cdot \vphi(S((\id\ten\vphi)(\Gam(x_i^*)(1\ten a)))^*)\\
&=\sum_{i=1}^n\overline{(y_i\cdot\vphi)^o((\id\ten\vphi)(\Gam(x_i^*)(1\ten a)))}\\
&=\sum_{i=1}^n\overline{\vphi(((y_i\cdot\vphi)^o\ten\id)(\Gam(x_i^*))a)}\\
&=\sum_{i=1}^n \la\Lphi(((y_i\cdot\vphi)^{o^*}\ten\id)\Gam(x_i)),\Lphi(a)\ra\\
&=\sum_{i=1}^n \la((y_i\cdot\vphi)^{o^*}\ten\id)(W^*)\Lphi(x_i),\Lphi(a)\ra\\
&=\sum_{i=1}^n \la\lm((y_i\cdot\vphi)^o)^*\Lphi(x_i),\Lphi(a)\ra\\
&=\sum_{i=1}^n \la\lm(y_i\cdot\vphi)\Lphi(x_i),\Lphi(a)\ra.\\
\end{align*}
This establishes the first formula. The second follows from the general relation $((y\cdot \vphi)^{o^*}\ten\id)\Gam(x)=(\vphi\ten\id)((S(y)\ten 1)\Gam(x))$ valid for all $x,y\in\D(S)$, as is easily verified.
\end{proof}

\begin{lem}\label{l:antipodeformula1} Let $\G$ be a compact quantum group, and $x\in\mc{A}$. Then $x\cdot\vphi\in\LOQs$ with $(x\cdot\vphi)^o=S(x)^*\cdot\vphi$.
\end{lem}

\begin{proof} First note that $\vphi=\vphi\circ S$ on $\mc{D}(S)$. Then for $y\in\mc{D}(S)$ we have
\begin{align*}\la S(x)^*\cdot\vphi,y\ra&=\vphi(yS(x)^*)=\overline{\vphi(S(x)y^*)}=\overline{\vphi(S(x)S(S(y)^*))}\\
&=\overline{\vphi(S(S(y)^*x))}=\overline{\vphi(S(y)^*x)}=\overline{\la x\cdot\vphi,S(y)^*\ra}.\end{align*}
Thus, $x\cdot\vphi\in\LOQs$ with $(x\cdot\vphi)^o=S(x)^*\cdot\vphi$.
\end{proof}

\begin{lem}\label{l:matrixunits} Let $\G$ be a compact quantum group. Then
$$\{d_\al\sqrt{\lm_i^\al\lm_j^\al}\lm(\vphi_{ij}^\al)\mid\al\in\Irr, 1\leq i,j\leq n_\al\}$$
forms a set of matrix units for the von Neumann algebra $\LIQH$. In particular, for every $\xi\in\LTQ$ we have
\begin{equation*}\xi=\sum_{\al\in\textnormal{Irr}(\G)}d_\al \Lm_\vphi(\chi_q^\al)\star\xi\end{equation*}
where the sum converges in $\LTQ$.
\end{lem}

\begin{proof} Let $e_{ij}^\al:=d_\al\sqrt{\lm_i^\al\lm_j^\al}\lm(\vphi_{ij}^\al)$ for $1\leq i,j\leq n_\al$ and $\al\in\Irr$. By Lemma \ref{l:antipodeformula1}, $\vphi_{ij}^\al\in\LOQs$ and $\vphi_{ij}^{\al^o}=S(u_{ij}^\al)^*\cdot\vphi=u_{ji}^\al\cdot\vphi=\vphi_{ji}^\al$.
Since $\lm$ is involutive on $\LOQs$ \cite[Proposition 2.4]{KV2}, we have $e_{ij}^{\al*}=e_{ji}^\al$. Then
\begin{align*}
\la\vphi_{ij}^\al\star\vphi_{kl}^\be,u_{mn}^{\gamma^*}\ra&=\sum_{r=1}^{n_\gamma}\la\vphi_{ij}^\al\ten\vphi_{kl}^\be,u_{mr}^{\gamma^*}\ten u_{rn}^{\gamma^*}\ra=\sum_{r=1}^{n_\gamma}\vphi(u_{mr}^{\gamma^*}u_{ij}^\al)\vphi(u_{rn}^{\gamma^*}u_{kl}^\be)\\
&=\frac{\delta_{\al\gamma} \delta_{\be\gamma}\delta_{im}\delta_{jk}\delta_{nl}}{\lm_k^\al \lm_i^\al d_\al^2}\\
&=\frac{\delta_{\al\be}\delta_{kj}}{\lm_k^\al d_\al}\la\vphi_{il}^\al,u_{mn}^{\gamma^*}\ra.
\end{align*}
It follows that $e_{ij}^\al e_{kl}^\be=d_\al d_\be\sqrt{\lm_i^\al\lm_j^\al}\sqrt{\lm_k^\be\lm_l^\be}\lm(\vphi_{ij}^\al\star\vphi_{kl}^\be)=\delta_{\al\be}\delta_{kj}e_{il}^\al$.
By above, we know that $\sum_{i=1}^{n_\al}e_{ii}^\al=d_\al\lm(\vphi^\al_q)\in\mc{Z}(\LIQH)$, hence $z_\al:=d_\al\lm(\vphi^\al_q)$ is a central projection in
$\LIQH$ acting as the identity on the factor $\{e^\al_{ij}\mid 1\leq i,j,\leq n\}\cong M_{n_\al}(\C)$. Thus,
$$\LIQH=\bigoplus_{\al\in\Irr}z_\al\LIQH\cong\bigoplus_{\al\in\Irr} M_{n_\al}(\C).$$
\end{proof}

We now show that the central vectors in $\LTQ$ are precisely the span of the characters, generalizing the well-known fact for compact groups.

\begin{prop}\label{p:fixedpoints} Let $\G$ be a compact quantum group. Then
$$Z\LTQ=\overline{\textnormal{span}}\{\Lm_\vphi(\chi^\al)\mid\al\in\textnormal{Irr}(\G)\}.$$
\end{prop}

\begin{proof} Since $\vphi^o=\vphi=\vphi^*$, by Lemma \ref{l:conj-formula} we have
\begin{align*}\la\beta_2 (\vphi)\Lm_\vphi(u^\al_{kl}),\Lm_\vphi(u^\be_{ij})\ra&=\sum_{n,m=1}^{n_\al}\la\Lphi((\vphi\ten\id)(u^{\al*}_{ln}u^\al_{km}\ten u^\al_{mn})),\Lphi(u^\be_{ij})\ra\\
&=\sum_{n,m=1}^{n_\al}\vphi((u^\al_{ln})^*u^\al_{km})\vphi((u^\be_{ij})^*u^\al_{mn})\\
&=\delta_{\al\be}\delta_{kl}\delta_{ij}\frac{1}{\lm^{\al}_i\lm^{\al}_k d_\al^2}\\
&=\delta_{kl}\frac{1}{\lm^\al_k d_\al}\la\Lm_{\vphi}(\chi^\al),\Lm_{\vphi}(u^\be_{ij})\ra.
\end{align*}
By density of irreducible coefficients, we obtain
\begin{equation}\label{conj}\beta_2 (\vphi)\Lm_\vphi(u^\al_{kl})=\frac{\delta_{kl}}{\lm_k^\al d_\al}\Lm_\vphi(\chi^\al).\end{equation}

Lemma \ref{l:matrixunits} implies that $\beta_2 (\vphi)$ is a self-adjoint idempotent, so
$\beta_2 (\vphi)$ is the orthogonal projection onto its fixed points, namely $Z\LTQ$. Equation (\ref{conj}) implies that $\beta_2 (\vphi)\Lm_\vphi(\chi^\al)=\Lm_\vphi(\chi^\al)$ for all $\alpha\in\Irr$ so that $\overline{\textnormal{span}}\{\Lm_\vphi(\chi^\al)\mid\al\in\textnormal{Irr}(\G)\}\subseteq Z\LTQ$.
Conversely, if $\xi\in Z\LTQ$ and $\la\xi,\Lm_{\vphi}(\chi^\al)\ra=0$
for all $\al$, then
$$\la\xi,\Lm_\vphi(u^\be_{kl})\ra=\la\xi,\beta_2 (\vphi)\Lm_\vphi(u^\be_{kl})\ra=\delta_{kl}\frac{1}{\lm^\al_k d_\al}\la\xi,\Lm_\vphi(\chi^\al)\ra=0$$
for all $\be\in\Irr$, $k,l=1,...,n_{\be}$. By density, we must have $\xi=0$.
\end{proof}

\begin{prop}\label{p:ZL2(G)}
 Let $\G$ be a compact quantum group. Then
$$\mathcal{Z}(\LTQ)=\overline{\textnormal{span}}\{\Lm_\varphi(\chi_q^\al)\mid\al\in\textnormal{Irr}(\G)\}.$$
\end{prop}

\begin{proof} Since $\{\vphi_q^\al\mid\al\in\Irr\}\subseteq\mathcal{Z}(\LOQ)$, it follows from the definition of convolution in $\LTQ$ that
$\overline{\text{span}}\{\Lm(\chi_q^\al)\mid\al\in\Irr\}\subseteq\mathcal{Z}(\LTQ)$. Conversely, suppose that $\xi\in\mathcal{Z}(\LTQ)$. Since $b(\LTQ)$ is dense
in $\LOQ$, it follows that $b(\xi)\in\mathcal{Z}(\LOQ)$, which in turn makes $\lm(b(\xi))\in\mathcal{Z}(\LIQH)$, so that
$d_\al\lm(\vphi_q^\al)\lm(b(\xi))=z_\al\lm(b(\xi))=c_\al z_\al=c_\al d_\al\lm(\vphi_q^\al)$ for each $\al\in\Irr$. By injectivity of $\lm$, we obtain
$d_\al\vphi_q^\al \star b(\xi)=c_\al d_\al \vphi_q^\al$. But $\vphi_q^\al=b(\Lphi(\chi_q^\al))$ so injectivity of $b$ implies
$d_\al \Lphi(\chi_q^\al)\star\xi=c_\al d_\al \Lphi(\chi_q^\al)$. Thus, by Lemma \ref{l:matrixunits}
\begin{equation}\label{eq:norm-2}
\xi=\sum_{\al\in\Irr}d_\al \Lm_\vphi(\chi_q^\al)\star\xi=\sum_{\al\in\Irr}c_\al d_\al \Lm_\vphi(\chi_q^\al)
\end{equation}
where the series converges in $\LTQ$.
\end{proof}

\begin{remark} Let $\beta_2':\LOQ\rightarrow\BLTQ$ be the representation defined by
$$\beta_2'(f)=U\beta_2(f) U^*=(f\ten\id)((1\ten U)W(1\ten U^*)W), \ \ \ f\in\LOQ.$$
Then $\beta_2'$ is another ``conjugation representation'' on $\LTQ$ whose fixed points are precisely $\mc{Z}(\LTQ)$. Thus, for non-Kac compact quantum groups, the conjugation representations $\beta_2$ and $\beta_2'$ distinguish the central vectors from the center of the Banach algebra $\LTQ$.

In the group setting, $W$ and $(1\ten U)W(1\ten U^*)$ belong to $\LI\oten VN(G)$ and $\LI\oten VN(G)'$, respectively, and therefore commute. Hence, for $f\in\LO$, we have
$$\beta_2(f)=\beta_2'(f)=\int_Gf(s)\lm(s)\rho(s)ds,$$
where $\lm$ and $\rho$ are the left and right regular representations of $G$, respectively, and $ds$ is the normalized Haar measure on $G$.
\end{remark}

We now establish the corresponding density theorem at the level of $\LIQ$.

\begin{thrm}\label{t:Wo} Let $\G$ be a compact quantum group. Then
$$Z\LIQ=\{\chi^\al\mid\al\in\Irr\}''.$$
\end{thrm}

\begin{proof} The inclusion $\{\chi^\al\mid\al\in\Irr\}''\subseteq Z\LIQ$ is clear. Let $x\in Z\LIQ$. Then for any $y\in\LIQ$
\begin{align*}\la\beta_2(\vphi)\Lphi(x),\Lphi(y^*)\ra&=\la W\sigma V\sigma\Lphis(1\ten x),\Lphis(1\ten y^*)\ra\\
&=\la\sigma\Lphis(\Gamma(x)),\Lphis(\Gamma(y^*))\ra\\
&=\vphi\ten\vphi(\Gam(y)\Sigma\Gam(x))=\vphi\ten\vphi(\Gam(yx))=\vphi(yx)\\
&=\la\Lphi(x),\Lphi(y^*)\ra.\end{align*}
It follows that $\beta_2(\vphi)\Lphi(x)=\Lphi(x)$. Hence $\Lphi(Z\LIQ)\subseteq Z\LTQ$ by Proposition \ref{p:fixedpoints}, and $Z\LTQ=\overline{\Lphi(Z\LIQ)}^{\norm{\cdot}}$, as the reverse inclusion is clear.

Note that equations (\ref{e:N}) and (\ref{eq:inner-on-characters}) entail the traciality of $\vphi$ on the von Neumann algebra $\{\chi^\al\mid\al\in\Irr\}''$. In particular, the map $Z\LTQ\ni \Lphi(x)\mapsto \Lphi(x^*)\in Z\LTQ$ is an isometry. Given, $x,y\in Z\LIQ$, take sequences $(x_n)$ and $(y_m)$ in $\mathrm{span}\{\chi^\al\mid\al\in\Irr\}$ such that $\Lphi(x_n)\rightarrow \Lphi(x)$ and $\Lphi(y_m)\rightarrow\Lphi(y)$. Then
\begin{align*}\vphi(x^*y)&=\la\Lphi(y),\Lphi(x)\ra=\lim_n\la\Lphi(y_n),\Lphi(x)\ra=\lim_n\lim_m\la\Lphi(y_n),\Lphi(x_m)\ra\\
&=\lim_n\lim_m\vphi(x_m^*y_n)=\lim_n\lim_m\vphi(y_nx_m^*)=\lim_n\lim_m\la \Lphi(x_m^*),\Lphi(y_n^*)\ra\\
&=\lim_n\la\Lphi(x^*),\Lphi(y_n^*)\ra=\vphi(yx^*).\end{align*}
Thus, $\vphi$ is a faithful trace on $Z\LIQ$, so there is a unique conditional expectation $E:Z\LIQ\rightarrow\{\chi^\al\mid\al\in\Irr\}''$ satisfying $\Lphi(E(x))=\beta_2(\vphi)\Lphi(x)=\Lphi(x)$, $x\in Z\LIQ$. Then $E(x)=x$, and $Z\LIQ\subseteq\{\chi^\al\mid\al\in\Irr\}''$.
\end{proof}

In a previous version of this paper \cite{AC}, the authors claimed that $ZC(\G)$ is the norm closed linear span of the characters in $C(\G)$ for any compact quantum group $\G$ (see \cite[Corollary 3.8]{AC}). However, the proof contains an error. We thank Matthew Daws, Jacek Krajczok and Christian Voigt for bring this to our attention. At this time, the authors do not know whether the analogous norm density result for characters in $ZC(\G)$ is valid for arbitrary compact quantum groups $\G$. Lemeux established said result for compact Kac algebras in \cite[Theorem 1.4]{Lemeux}, which partially answered an open question of Woronowicz (see \cite[Proposition 5.11]{Wo}).

For compact groups $G$, the standard proof that $\mc{Z}(\LO)$ is the closed linear span of the characters utilizes a central bounded approximate identity (BAI) for $\LO$. A similar argument applies for any compact quantum group $\G$ for which $\LOQ$ has a BAI $(f_i)$ in $\mathrm{span}\{\vphi_q^\al\mid\al\in\Irr\}$, which is bounded in the completely bounded multiplier norm, i.e., the maps $\LOQ\ni g\mapsto f_i\star g\in \LOQ$ are uniformly completely bounded. Any $\G$ whose dual has the central almost completely positive approximation property (ACPAP) in the sense of \cite[Definiton 3]{CFY} has this property. Thus, \begin{equation}\label{e:equality}\mc{Z}(\LOQ)=\overline{\mathrm{span}}\{\vphi_q^\al\mid\al\in\Irr\}\end{equation}
for any compact $\G$ with the central ACPAP. By \cite[Theorem 25]{CFY}, this includes $SU_q(2)$, $q\in[-1,0)\cup(0,1]$, as well as any free orthogonal and unitary quantum groups $O_F^+$ and $U_F^+$, for any parameter matrix $F\in GL(n,\C)$ (see \cite[\S1.4,\S4.2]{CFY}, for instance). We conjecture that (\ref{e:equality}) is valid for arbitrary compact $\G$. We now provide support for the conjecture by showing that it holds for arbitrary compact Kac algebras. In turn, we generalize a result of Mosak \cite[Proposition 1.5 (i)]{Mosak}.

\begin{thrm}\label{t:extension} Let $\G$ be a compact Kac algebra. Then $\beta_2(\vphi):\LTQ\rightarrow\LTQ$ extends to a completely contractive projection $\beta_1:\LOQ\rightarrow\mc{Z}(\LOQ)$ satisfying
$$\beta_1(f\star g)=f\star\beta_1(g), \ \ \ f\in\mc{Z}(\LOQ), \ g\in\LOQ.$$
\end{thrm}

\begin{proof} The argument in \cite[Lemma 3.2]{RX} shows that the map $\Phi:\LIQ\oten\LIQ\rightarrow\LIQ$ given by
 $$\Phi(X)=(\om_{\Lphi(1)}\ten\id)W^*(U^*\ten 1)X(U\ten1)W, \ \ X\in\LIQ\oten\LIQ$$
is a normal, unital, completely positive left inverse to $\Gamma$ satisfying $\Gamma\circ\Phi=(\Phi\ten\id)(\id\ten\Gamma)$. Moreover, since $W\in\LIQ\oten\LIQH$ and $U\Lphi(1)=\Lphi(1)$, we have
$$\vphi(\Phi(x\ten y))=(\om_{\Lphi(1)}\ten\vphi)((U^*xU\ten1)W^*(1\ten y)W)=\vphi(x)\vphi(y)$$
for all $x,y\in\LIQ$. By normality it follows that $\vphi\ten\vphi=\vphi\circ\Phi=\vphi\ten\vphi\circ\Gamma\circ\Phi$, so that $\Gamma\circ\Phi$ is a normal
conditional expectation onto $\Gamma(\LIQ)$ preserving $\vphi\ten\vphi$.

By Lemma \ref{l:conj-formula}, the map $b\circ\beta_2(\vphi)\circ a:\mc{I}_\vphi\rightarrow\LOQ$ satisfies $b\circ\beta_2(\vphi)\circ a(x\cdot\vphi)=\sum_{i=1}^n(y_i\cdot\vphi)\star(x_i\cdot\vphi)$
for $x\in\mc{A}$, where $\Gamma(x)=\sum_{i=1}^nx_i\ten y_i$. Recalling that $\vphi$ is invariant under the unitary antipode, for $y\in\LIQ$ we have
\begin{align*}\la b\circ\beta_2(\vphi)\circ a(x\cdot\vphi),y\ra&=\sum_{i=1}^n\la(y_i\cdot\vphi)\star(x_i\cdot\vphi),y\ra=\sum_{i=1}^n\la(y_i\cdot\vphi)\ten(x_i\cdot\vphi),\Gamma(y)\ra\\
&=\sum_{i=1}^n\vphi\ten\vphi(\Gamma(y)(y_i\ten x_i))=\vphi\ten\vphi(\Gamma(y)\Sigma\Gamma(x))\\
&=\vphi\ten\vphi(\Gamma(R(x))(R\ten R)(\Gam(y)))\\
&=\vphi\ten\vphi(\Gamma\circ\Phi(\Gamma(R(x))(R\ten R)(\Gam(y))))\\
&=\vphi\ten\vphi(\Gamma(R(x))\Gamma\circ\Phi((R\ten R)(\Gam(y))))\\
&=\vphi\ten\vphi(\Gamma(R(x)\Phi((R\ten R)(\Gam(y)))))\\
&=\vphi(R(x)\Phi((R\ten R)(\Gam(y))))\\
&=\vphi(R(\Phi((R\ten R)(\Gam(y))))x)\\
&=\la x\cdot\vphi,R\circ\Phi\circ(R\ten R)\circ\Gamma(y)\ra.
\end{align*}
Since the $\mathrm{span}\{x\cdot\vphi\mid x\in\mc{A}\}$ is dense in $\LOQ$, it follows that the map $b\circ\beta_2(\vphi)\circ a$ has a completely contractive extension to a map
$\beta_1:\LOQ\rightarrow\LOQ$ whose adjoint $\beta_1^*:\LIQ\rightarrow\LIQ$ is given by $\beta_1^*(y)=R\circ\Phi\circ(R\ten R)\circ\Gamma(y)$ for $y\in\LIQ$. Moreover,
since $\beta_2(\vphi)a(x\cdot\vphi)=ab(\beta_2(\vphi)a(x\cdot\vphi))$ for $x\in\mc{A}$, and $\beta_2(\vphi)$ is idempotent, we have
$$\beta_1(x\cdot\vphi)=b\beta_2(\vphi)\beta_2(\vphi)a(x\cdot\vphi)=b\beta_2(\vphi)a\circ b\beta_2(\vphi)a(x\cdot\vphi)=\beta_1\circ\beta_1(x\cdot\vphi)$$
which, by density, implies that $\beta_1$ is also idempotent.

To establish the module property, fix $g\in\LOQ$ and $f\in\mc{Z}(\LOQ)$. Then
\begin{align*}R\ten R(\Gamma(f\star x))&=R\ten R(\Gamma((\id\ten f)\Gam(x)))\\
&=R\ten R((\id\ten\id\ten f)(\Gam\ten\id)(\Gam(x)))\\
&=(\id\ten\id\ten f)((R\ten R)(\Gam\ten\id)(\Gam(x))\\
&=(\id\ten\id\ten f)(\Sigma\Gam\ten\id)(R\ten\id)(\Gam(x)) \ \ \ (\textnormal{equation (\ref{e:antipode})})\\
&=(\id\ten\id\ten R_*(f))(\Sigma\Gam\ten\id)(R\ten R)(\Gam(x))\\
&=(\id\ten\id\ten R_*(f))(\Sigma\Gam\ten\id)(\Sigma\Gam(R(x)))\\
&=(\id\ten\id\ten R_*(f))(\id\ten\Sigma\Gam)(\Sigma\Gam(R(x))) \ \ \ (\textnormal{co-associativity})\\
&=(\id\ten\id\ten R_*(f))(\id\ten\Gam)(\Sigma\Gam(R(x))) \ \ \ (f\in\mc{Z}(\LOQ)).\\
\end{align*}
Applying the map $\Phi$, we obtain
\begin{align*}\Phi(R\ten R(\Gamma(f\star x)))&=(\id\ten R_*(f))(\Phi\ten\id)(\id\ten\Gam)(\Sigma\Gam(R(x)))\\
&=(\id\ten R_*(f))\Gam(\Phi(\Sigma\Gam(R(x))))\\
&=(\id\ten R_*(f))\Gam(\Phi(R\ten R(\Gam(x))))\\
&=R_*(f)\star\Phi(R\ten R(\Gam(x))).\\
\end{align*}
Therefore
$$R\circ\Phi\circ(R\ten R)\circ\Gam(f\star x)=R\circ\Phi\circ(R\ten R)\circ\Gam(x)\star f=f\star R\circ\Phi\circ(R\ten R)\circ\Gam(x)$$
as $f$ is central. The pre-adjoint of the above relation yields the desired module property.

Now, let $f\in\LOQ$. Taking a sequence $(\Lphi(x_n))$ in $\LTQ$ such that $b(\Lphi(x_n))$ converges to $f$, it follows that $\beta_1(b(\Lphi(x_n)))$ converges to $\beta_1(f)$. But $\beta_1(b(\Lphi(x_n)))=b(\beta_2(\vphi)\Lphi(x_n))\in b(\mc{Z}(\LTQ))$, so that $\beta_1(f)\in\overline{b(\mc{Z}(\LTQ))}$, which, by Proposition \ref{p:ZL2(G)}, is contained in $\overline{\textnormal{span}}\{\vphi^\al\mid\al\in\Irr\}\subseteq\mc{Z}(\LOQ)$.

Conversely, let $f\in\mc{Z}(\LOQ)$. By Proposition \ref{p:fixedpoints} we have
$\beta_1(\vphi^\al)=\vphi^\al$ for all $\al\in\Irr$, so the module property of $\beta_1$ entails
$$\beta_1(f)\star\vphi^\al=\beta_1(f\star\vphi^\al)=f\star\beta_1(\vphi^\al)=f\star\vphi^\al, \ \ \ \al\in\Irr.$$
Hence, by Lemma \ref{l:matrixunits}, for every $\xi\in\LTQ$,
$$\lm(f)\xi=\sum_{\al\in\textnormal{Irr}(\G)}d_\al \lm(f\star\vphi^\al)\xi=\sum_{\al\in\textnormal{Irr}(\G)}d_\al \lm(\beta_1(f)\star\vphi^\al)\xi=\lm(\beta_1(f))\xi.$$
By injectivity of $\lm$ we then get $f=\beta_1(f)$ so that $\mc{Z}(\LOQ)=\beta_1(\LOQ)$.
\end{proof}

\begin{cor}\label{c:ZL1(G)=B1}
Let $\G$ be a compact Kac algebra. Then
$$\mc{Z}(\LOQ)=\overline{\mathrm{span}}\{\vphi^\al\mid\al\in\textnormal{Irr($\G$)}\}.$$
\end{cor}

\begin{remark} In regards to the conjecture (\ref{e:equality}), it would be interesting to first examine the class $SU_q(2n+1)$, $n\geq 1$, as their duals have recently been shown to exhibit central property $(T)$ \cite[Corollary 8.9]{A}.\end{remark}

\section*{Acknowledgements}

This project was initiated at the Fields Institute during the Thematic Program on Abstract Harmonic Analysis, Banach and Operator Algebras. We are grateful to the Institute for its kind hospitality. We also thank the anonymous referee whose comments helped improve the presentation of the paper.


\begin{thebibliography}{10}


\bibitem{AC} M. Alaghmandan and J. Crann,
\e{Character density in central subalgebras of compact quantum groups}. 
Canadian Math. Bull. {\bf 60} (2017), no. 3, 449--461.

\bibitem{A} Y. Arano,
\e{Unitary spherical representations of Drinfeld doubles}.
J. Reine Angew. Math. in press, 2016. DOI: 10.1515/crelle-2015-0079.

\bibitem{Teo} T. Banica,
\e{Th\'{e}orie des repr\'{e}sentations du groupe quantique compact libre $O(n)$}.
C. R. Acad. Sci. Paris S\'{e}r. I Math. {\bf 322}(1996), no.~3, 241--244.

\bibitem {BT} E. B\'{e}dos and L. Tuset,
\e{Amenability and co-amenability for locally compact quantum groups}.
Internat. J. Math. {\bf 14}(2003), no.~8, 865--884.

\bibitem{B} M. Brannan,
\e{Approximation properties for free orthogonal and free unitary quantum groups}.
J. Reine Angew. Math. {\bf 672}(2012), 223--251.

\bibitem{CFY} K. De Commer, A. Freslon and M. Yamashita,
\e{CCAP for universal discrete quantum groups}.
With an appendix by Stefaan Vaes. Comm. Math. Phys. {\bf 331}(2014), no.~2, 677--701.

\bibitem{D1} M. Daws,
\e{Operator biprojectivity of compact quantum groups}.
Proc. Amer. Math. Soc. {\bf 138}(2010), no.~4, 1349--1359.

\bibitem{ES} M. Enock and J. M. Schwartz,
\e{Kac Algebras and Duality of Locally Compact Groups}.
Springer--Verlag, Berlin, 1992.



\bibitem{K} J. Kustermans,
\e{Locally compact quantum groups in the universal setting}.
Internat. J. Math. {\bf 12}(2001), no.~3, 289--338.

\bibitem{KV1} J. Kustermans and S. Vaes,
\e{Locally compact quantum groups}.
Ann. Sci. \'{E}cole Norm. Sup. (4) {\bf 33}(2000), no.~6, 837--934.

\bibitem{KV2} J. Kustermans and S. Vaes,
\e{Locally compact quantum groups in the von Neumann algebraic setting}.
Math. Scand. {\bf 92}(2003), no.~1, 68--92.

\bibitem{Lemeux} F. Lemeux,
\e{Haagerup approximation property for quantum reflection groups}.
Proc. Amer. Math. Soc. {\bf 143}(2015), no.~5, 2017--2031.

\bibitem{Mosak} R.~D. Mosak,
\e{The $L^1$- and $C^*$-algebras of $[FIA]^-_B$ groups, and their representations}.
Trans. Amer. Math. Soc. {\bf 163}(1972), 277--310.

\bibitem{NT2} S. Neshveyev and L. Tuset,
\e{Compact Quantum Groups and Their Representation Categories}.
Cours Sp\'{e}cialis\'{e}s--Collection SMF, 20, 2014.

\bibitem{RX} Z.-J. Ruan and G. Xu,
\e{Splitting properties of operator bimodules and operator amenability of Kac algebras}.
Operator theory, operator algebras and related topics (Timisoara, 1996), 193-216, Theta Found., Bucharest, 1997.

\bibitem{V} S. Vaes,
\e{Locally Compact Quantum Groups}.
Ph. D. thesis, K.U. Leuven, 2000.

\bibitem{Wo} S.~L. Woronowicz, S. L.
\e{Compact matrix pseudogroups}.
Comm. Math. Phys. {\bf 111}(1987), no.~4, 613--665.

\end{thebibliography}
\end{document}